\documentclass[a4paper,11pt,draft]{amsart}
\usepackage{eucal}
\usepackage{amsmath,amsthm,amssymb}
\usepackage{amsfonts}
\usepackage{latexsym}
\usepackage[dvips]{graphics}
\usepackage{amsrefs}
\usepackage{mathrsfs}
\theoremstyle{plain}
\newtheorem{thm}{Theorem}[section]
\newtheorem{prop}[thm]{Proposition}
\newtheorem{lemma}[thm]{Lemma}
\newtheorem{cor}[thm]{Corollary}

\theoremstyle{definition}
\newtheorem{defn}[thm]{Definition}
\newtheorem{exa}[thm]{Example}

\theoremstyle{remark}
\newtheorem{rem}[thm]{Remark}

\numberwithin{equation}{section}

\makeatletter
\@namedef{subjclassname@2020}{\textup{2020} Mathematical Subject Classification}
\makeatother
\title{ Hopfological invariants for tame subextensions} 
\date{\today} 
\author{Mariko Ohara}
\address{Center for Liberal Arts and Sciences, Faculty of Engineering, Toyama Prefectural University, 5180, Kurokawa, Imizu City, Toyama Prefecture, JAPAN, 939-0398. 
\\ Tel : +81-766-56-7500.
}  
\email{primarydecomposition@gmail.com}
\newcommand{\Ker}{{\mathrm{Ker}}}
\newcommand{\Hom}{\mathrm{Hom}}

\newcommand{\LMod}{\mathrm{LMod}}

\newcommand{\End}{{\mathrm{End}}}

\usepackage{enumerate}
\usepackage{array}
\usepackage[all]{xy}
\usepackage{delarray}
\ifx\undefined\bysame
\newcommand{\bysame}{\leavemode\hbox to3em{\hrulefill}\,}
\fi
\begin{document}
\thispagestyle{empty}

\begin{abstract}
Let $H$ be a finite dimensional Hopf algebra over a field $K$. In this paper, we study when an $H$-extension becomes a tame $H$-extension by calculating Hopfological homology and Hopf-cyclic homology. 
In the (derived) category of $\mathcal{H}$-comodules for a Hopf algebra $\mathcal{H}$ studied by Farinati~\cite{Fari}, we take Hopf subalgebra $H < \mathcal{H}$ and a certain order $\mathcal{A} \subset H$. 
We see the behavior of Hopfological homology for a tame $\mathcal{A}$-subextension $S/R$ in terms of the surjectivity of trace map and of cyclic modules, which induce Hopf-cyclic homology, for Hopf-Galois extensions with $H$ in terms of relative Hopf modules. 


\end{abstract}

\maketitle


\section{Introduction} 
Let $R$ be a commutative ring and $H$ an $R$-Hopf algebra. Recall that a finite $R$-algebra $S$, i.e., an $R$-algebra that is a finitely generated projective $R$-module, which admits a left  $H$-action is said to be an $H$-extension if we have $S^H=R$ for the fixed points of $S$ with respect to the action $H$ on $S$. We say that $S$ is a tame $H$-extension of $R$ if $S$ is an $H$-extension with $IS=R$ for the set of left integrals $I=H^H$ and is a faithful left $H$-module with $rank_R(S)=rank_R(H)$ as in ~\cite{Childs1}. 

A notion of Hopf-Galois extensions are arising from the notion of Galois extension of rings, and it deeply related to tame $H$-extensions.  
Let $H$ be a finite cocommutative $R$-Hopf algebra. A finite commutative $R$-algebra $S$ is a Hopf-Galois extension of $R$ if $S$ is a left $H$-module algebra, and the $R$-module homomorphism $j : S \otimes_R H \to \End_R(S) ; j(s \otimes h)(t)=sh(t)$ is an isomorphism. If the Hopf algebra $H$ is finite $R$-algebra, the $H$-module structure gives an $H^\ast=\Hom_R(H, R)$-comodule structure and $j$ induces a map $j': S \otimes_R S \to S \otimes_RH^\ast$, and $j$ becomes an isomorphism if and only if the induced map $j': S \otimes_R S \to S \otimes_R H^\ast$ is an isomorphism.

When $S/R$ is a finite field extension with the Galois group $G$, it is equivalent to say that $S/R$ is a Hopf-Galois extension with Hopf algebra $RG$, 
especially we have $S \# RG \cong \End_R(S) \cong S[G]$. Here, $S \# RG$ is the smash product over $R$ with respect to the Hopf algebra $RG$ and $S[G]$ is a group algebra whose underlying set is given by $\{ \Sigma_{\sigma \in G}s_{\sigma}\sigma | s_{\sigma} \in S\}$. 



For a Hopf algebra $H$ over a local field $K$ and a Hopf-Galois extension $L/K$ with $H$,  an associated order of $\mathcal{O}_L$ is defined as $\mathcal{A}=\{ h \in H | h x \in \mathcal{O}_L \, \text{for all} \, x \in \mathcal{O}_L \}$. 
If $\mathcal{A}$ in $H$ is a Hopf algebra, the integer ring $\mathcal{O}_L$ is a tame $\mathcal{A}$-extension of $\mathcal{O}_K$ if there exists a left integral $\lambda$ of $\mathcal{A}$ satisfying $\lambda \mathcal{O}_L=\mathcal{O}_K$
  by Childs~\cite{Childs1}. He proved that if $\mathcal{O}_L$ is a tame $\mathcal{A}$-extension of $\mathcal{O}_K$, then $\mathcal{O}_L$ is a free $\mathcal{A}$-module of rank one. 
However, there are many wildly ramified Galois extensions $L/K$ whose valuation rings are free over their associated order $\mathcal{A}$ but $\mathcal{A}$ is not a Hopf order. For example, if $L$ is an abelian extension of $\mathbb{Q}$ with Galois group $G$, then $\mathcal{O}_L$ is free over the associated order $\mathcal{A}$ by Leopoldt's theorem but $\mathcal{A}$ may not be a Hopf order.



Let $L/K$ be a Hopf-Galois extension of local fields with Hopf algebra $H$ over $K$. Let $S$ be the the integral closure of the integer ring $R \subset K$. Then, if the associated order $\mathcal{A} \subset H$ of $S$ is a finite dimensional Hopf algebra over $R$, $S/R$ is a tame $\mathcal{A}$-extension. If $H$ is cocommutative and $rank_R(S)=rank_R(\mathcal{A})$, the tame $\mathcal{A}$-extension is equivalent to Hopf-Galois extension with $\mathcal{A}$. 
Since $R$ is a discrete valuation ring, if $\mathcal{A}$ is a finite Hopf algebra over $R$, it has a nonzero left integral. 
Note that the certain projectivity conditions of Hopf-Galois extensions are automatically satisfied in this setting.

From the point of view, we would like to know when the associated order $\mathcal{A} \subset H < \mathcal{H}$ becomes a Hopf algebra, and when $\mathcal{A}$ is a Hopf algebra, we would like to know how Hopf-Galois extension affects Hopf-cyclic homology. 

Jara and \c{S}tefan took a Hopf subalgebra $K$ of a Hopf algebra $H$ and consider the Hopf-cyclic homology with modular pair of induced module $\mathrm{Ind}^H_K(M)=H \otimes_K M$ in \cite{JS1}. 
They showed that the cyclic structure on the relative cyclic complex of $B \subset A$ induces a cyclic module, which depends only on $H$ and stable anti-Yetter Drinfeld module $A/ [A, B]$. 
This calculation is fruitful in respect of Hopf-Galois extension. 

In this short paper, we take a (Hopf) subalgebra $H \subset \mathcal{H}$. We would like to see when an $H$-extension becomes a tame $H$-extension in certain cases by using Hopfological homology as in \cite{Qi}, \cite{Fari}, and study the behavior of cyclic modules, which induce Hopf-cyclic homology as in \cite{Kaygun} and \cite{Ohara1} when an extension $S/R$ becomes a tame $H$-subextension $S/R$ in terms of the surjectivity of trace maps as follows. 


\begin{thm}
Let $R$ be a commutative local ring. Let $\mathcal{A}$ be a finite $R$-Hopf algebra and $S$ a faithful $\mathcal{A}$-module algebra which is a finite commutative $R$-algebra of $rank_R(S)=rank_R(\mathcal{A})$. Assume that $S^\mathcal{A}=R$. 
\begin{enumerate}[(i)]
\item if $R$ is a principal ideal domain, $S/R$ is a tame $\mathcal{A}$-extension if and only if the Hopfological homology of $S$ is equal to zero.   
\item If $R$ is a field $K$ and $S/K$ is a tame $\mathcal{A}$-extension,  the Hopfological homology of any relative $(S, \mathcal{A}^\ast)$-module is zero. 
Let $R$ be a field and $S$ a finite faithful $H$-extension with a Hopf algebra $H$ over $R$. 
Suppose that $H$ is a local cocommutative Hopf algebra over $K$ and $rank_R(H)=rank_R(S)$. 
Then $S/R$ is a Hopf-Galois extension with $H$ if and only if the Hopfological homology of $S$ is equal to zero. 
\end{enumerate}
\end{thm}

We use the tensor product over $R$. 
We let $B_n(S, M)=S^{\otimes n}\otimes_R M$ and denote by $B_\bullet (S, M)$ the simplicial module obtained by one sided bar construction with coefficients in $M$. 
We define $T_n(S, M)$ by $S^{\otimes n} \otimes_R M$ and the face, degeneracy, and cyclic operators on it as in Section 3. 
We take $M'$ to be a relative $(S, H^\ast)$-module. Since we assume that the antipode of $H$ is bijective, we can associated the left $H^\ast$-comodule $M$ to the right $H$-comodule $M'$. Thus, eventually we can assume that $M$ is a stable left-left anti-Yetter-Drinfeld module over $H$ and use the (co)cyclic theory as in \cite[Section 2, Section 3]{HKRS}.  

Before assuming $M$ to be an anti-Yetter-Drinfeld module, we show the following. 

\begin{thm}
Let $R$ be a commutative local ring. Let $\mathcal{A}$ be a finite $R$-Hopf algebra with bijective antipode and $S$ a faithful $\mathcal{A}$-module algebra which is a finite commutative $R$-algebra. 
Assume $S^\mathcal{A}=R$. Assume that $M$ is an $S \# \mathcal{A}$-module. 
\begin{enumerate}[(i)]
\item If $j : S \# H \to E=End_R(S)$ is an isomorphism, then $B_\bullet(S, M)$ is isomorphic to the degree $+1$ shift of $B_\bullet(S, M^\mathcal{A})$, i.e., $B_\bullet(S, M) \cong B_\bullet(S, M^\mathcal{A})[1]$. 
\item Assume further that $\mathcal{A}$ is a finite local cocommutative $R$-Hopf algebra with bijective antipode and $S$ satisfies $rank_R(S)=rank_R(\mathcal{A})$. 
Then, if the Hopfological homology of $S$ is equal to zero, then $B_\bullet(S, M) \cong B_\bullet(S, M^\mathcal{A})[1]$. 
\item Let $K$ be a field and $S$ an $\mathcal{H}$-module commutative algebra over $K$. For a finite Hopf subalgebra $H \subset \mathcal{H}$ with bijective antipode, assume that $S$ is a finite faithful $H$-extension of  $R=S^{coH^\ast}$ by restricting action of $\mathcal{H}$. 
Suppose that $S$ is a faithfully flat left $R$-module and $S/R$ is a Hopf-Galois extension with $H$. 
Then, $T_\bullet(S, M) \cong T_\bullet(S, M^\mathcal{A})[1]$. 
\end{enumerate}
\end{thm}

\subsection*{Acknowledment} The author would like to thank Professor Takeshi Torii for suggestions and arguments on $K(n)$-local spectra, which inspired the author to study tame extensions.

%

\section{Preliminary}

Let $K$ be a field and $\mathcal{H}$ a Hopf algebra over $K$. 

For a coFrobenius Hopf algebra $\mathcal{H}$ over a field $K$, the category of right $\mathcal{H}$-comodules, denoted by $\mathcal{M}^\mathcal{H}$, and its derived category is studied in \cite{Fari}. 

If $M$ is a left $\mathcal{H}$-comodule, then $M$ can be regarded as a right $\mathcal{H}^\ast$-comodule. Here $\mathcal{H}^\ast$ stands for the linear dual $\Hom_K(\mathcal{H}, K)$ of $\mathcal{H}$. 

Now, we assume that $H \subset \mathcal{H}$ is a finite dimensional Hopf subalgebra over $K$. 
For a Hopf algebra $H$ finite dimensional over a field $K$, it is a Frobenius algebra, so that a projective $H$-module is just an injective $H$-module. We remark that the category of left $H$-modules is equivalent to the category of right $H^\ast$-comodule in this case. 
We say that $\lambda \in H$ is the left integral if it satisfies $h\lambda = \varepsilon(h)\lambda$ for all $h \in H$. 
We assume that $H$ is not semisimple throughout this paper for the existence of nonzero left integral. 

In the case that $H$ is finite dimensional over a field $K$, Qi~\cite{Qi} defined the derived category of $A \# H$-modules by formally inverting stable equivalences and it becomes a triangulated category. 

In the (derived) category  of $A \# H$-modules, we have the invariant named Hopfological homology. 
We regard $A \# H$-modules as $H$-modules via the forgetfull functor $\mathrm{U} : \LMod_{A \# H} \to \LMod_H$. 
\begin{defn}[Hopfological homology]
Let $H$ be a finite-dimensional Hopf algebra over $K$ and $I \subset H$ the set of left integrals. For an $H$-module $V$, the Hopf homology is given by $\displaystyle{\frac{V^H}{I V}}$. 
\end{defn}
\begin{rem}
Similarly, the Hopfological homology for a right $H$-comodule $M$ is defined, as in ~\cite[Section 3.2]{Fari}, by $\displaystyle{H_0^H(M)=\frac{M^{coH}}{I M}}$. It becomes a  stable invariant in the triangulated category of right $H$-comodules.  
\end{rem}

Now, we recall the notion of $H$-extension, tame $H$-extension and Hopf-Galois extension, respectively. 

\begin{defn}[cf. \cite{Childs1}, Definition 2.7, Definition 13.1]\label{ext}
Let $R$ be an associated unital ring. 
\begin{itemize}
\item Let $H$ be a Hopf algebra over $R$ and $S$ a finite $R$-algebra which is also a left $H$-module algebra. If $R=S^H$, then we say that $S/R$ is an $H$-extension. 
\item (Hopf-Galois extension; original form) Let $R$ be a commutative ring and $H$ a finite cocommutative $R$-Hopf algebra. Let $S$ be a finite commutative $R$-algebra and left $H$-module algebra. Then $S/R$ is Hopf-Galois extension with $H$ if the $R$-module homomorphism $j : S \# H  \to \End_R(S)$ which is given by $j(s \otimes h)(t)=sh(t)$ is an isomorphism. Here the smash product $S \# H$ is taken over $R$.  
\item (Hopf-Galois extension; in the sence of principal homogenious space) Let $R$ be a commutative ring and $H$ a finite $R$-Hopf algebra. Let $S$ be a finite $R$-algebra and left $H$-module algebra. Then $S/R$ is Hopf-Galois extension with $H$ if and only if the map $\gamma : S \otimes_R S \cong S \otimes_R H^\ast ; s \otimes t \mapsto (s \otimes 1) \sigma (t)$ becomes bijective, where, the map $\sigma : S \to S \otimes_R H^\ast$ is the induced comodule structure. 
\item  Let $R$ be an associated ring. Let $S$ be a finite $R$-algebra. Assume that $S$ is a left $H$-module algebra and $S^H=R$. We say that $S/R$ is tame $H$-extension if it is an $H$-extension such that $rank_R(H)=rank_R(S)$, $S$ is a faithful $H$-module and $IS=R$, where $I=H^H$. 
\end{itemize}
\end{defn}



\begin{rem}
Let $R$ be a commutative algebra, $H$ a finite Hopf algebra over $R$ and $S$ a finite commutative $R$-algebra which is also a left $H$-module algera. Then $j : S \otimes_R H  \to \End_R(S)$ which is given by $j(s \otimes h)(t)=sh(t)$ is an isomorphism if and only if the map $\gamma : S \otimes_R S \cong S \otimes_R H^\ast ; s \otimes t \mapsto (s \otimes 1) \sigma (t)$ is an isomorphism, where, the map $\sigma : S \to S \otimes_R H^\ast$ is the induced comodule structure. 

\end{rem}

\begin{rem}
Note that if $S$ is commutative in a Hopf-Galois extension, $H^\ast$ is also commutative. 
\end{rem}

\begin{rem}
The condition of Galois extension $S/R$ of rings requires $S[G] \cong \End_R(S)$ instead of the bijectivitiy of  $\gamma : S \otimes_R S \to S \otimes_R RG^\ast ; x \otimes y \mapsto \Sigma x y_{(0)} \otimes y_{(1)}$. 
\end{rem}


\begin{exa}
Let $L/K$ be a Galois field extension with Galois group $G$. An $L$-vector space $M$ becomes a Hopf module $M \in \mathcal{M}^{KG}_L$, via an action of the Galois group $G$ on $M$, by semilinear automorphisms $\sigma \cdot (am)=\sigma (a)(\sigma \cdot m)$ for all $m \in M$, $a \in L$ and $\sigma \in G$. Galois descent says that such an action on $M$ to be obtained from a $K$-vector space by extending scalars. This is part of the fundamental theorem for Hopf modules. 
\end{exa}

\begin{exa}
Let $A=\oplus_{g \in G}A_g$ be a $K$-algebra graded by a group $G$. Then $A$ is naturally an $KG$-comodule algebra whose coinvariant subring is $B=A_e$. If $K$ is a field and $A_gA_h=A_{gh}$ for all $g, h \in G$, i.e., $A$ is strongly graded, then the Galois map $A \otimes_B A \to A \otimes_K KG$ is surjective. This implies that $A$ is an Hopf-Galois extension of $B$ by \cite[Theorem 1]{Schneider}. 
\end{exa}

\begin{lemma}
Let $R$ be a commutative ring and $S$ a finite commutative $R$-algebra.  
Let $H$ be a finite $R$-Hopf algebra. Then if the map $\gamma : S \otimes_R S \cong S \otimes_R H^\ast$ is a bijective, it is an algebra map. 
\end{lemma}
\begin{proof}
We have $\gamma (xx' \otimes yy')=\Sigma xx'y_{(0)}y'_{(0)}\otimes y_{(1)}y'_{(1)}=\gamma (x \otimes y)\gamma (x' \otimes y')$. 
\end{proof}

\begin{rem}[CH-tame and CEPT-tame]
Chinburg, Erez, Pappus and Taylor define an $H$-comodule algebra structure $\alpha : S \to S \otimes_R H$ to be tame if and only if there is an $H$-comodule map $g : H \to S$ such that $g(1)=1$. This is called CEPT-tameness. The tameness in Definition~\ref{ext} is called CH-tameness. 
\end{rem}



The surjectivity of trace map is closely related to the existence of total integral by virtue of Larson-Sweedler theorem.  

Recall that an $\mathcal{A}$-module is faithful if it satisfies the condition that an element $a \in \mathcal{A}$ such hat $as=0$ for all $s \in S$ is only $a=0$. 

\begin{prop}\label{D1}
Let $R$ be a local ring. Assume that $\mathcal{A}$ is a finite $R$-Hopf algebra and $S$ is a faithful $\mathcal{A}$-module algebra which is a finite $R$-algebra of $rank_R(S)=rank_R(\mathcal{A})$ and that $S^\mathcal{A}=R$. Then $S$ is a tame $\mathcal{A}$-extension of $R$ if and only if there exists left $\mathcal{A}$-module homomorphism  $g : \mathcal{A}^\ast \to S$ such that $g(1)=1$. Especially, if $R$ is a field and $S$ is a tame $\mathcal{A}$-extension of $R$, any relative $(S, \mathcal{A}^\ast)$-Hopf module is an injective $\mathcal{A}^\ast$-comodule. 
\end{prop}
\begin{proof}
By Larson-Sweedler theorem, we have the left $\mathcal{A}$-module $\mathcal{A}^\ast=\mathcal{A}t$. 
Since $R$ is a local ring, we take a nonzero left integral $\lambda \in \mathcal{A}$ such that $\lambda t=1$. If $S$ is a tame $\mathcal{A}$-extension, then $\lambda S=R$, so we have $z \in S$ such that $\lambda z=1$. Then we have an $\mathcal{A}$-module homomorphism $g : \mathcal{A}^\ast \to S$ by $g(t)=z$. Then $g(\lambda t)=g(1)=\lambda z=1$. 

Conversely if $g : \mathcal{A}^\ast \to S$ is an $\mathcal{A}$-module homomorphism with $g(1)=1$, then there exists an element in $R$-dual basis, which exists by finiteness, $g(t)=z \in S$ which gives $\lambda z = \lambda g(t)=g(\lambda t)=g(1)=1$ so the trace map is surjective. 

The last assertion follows from Doi's theorem~\cite[Chapter II, Theorem 1]{Doi2}; which says, for a Hopf algebra $\mathcal{A}$ over a field $K$, if there is a right $\mathcal{A}^\ast$-comodule map $\mathcal{A}^\ast \to S$, then any right relative $(S, \mathcal{A}^\ast)$-Hopf module is an injective $\mathcal{A}^\ast$-comodule. 
\end{proof}

\begin{rem}
Let $R$ is a local ring. If $f : R \to S$ is a flat algebra map. For a prime ideal $q \subset S$, assume that $p=f^{-1}(q)$ is nonzero, then $R \to S_q$ is faithfully flat. 
\end{rem}

\begin{rem}
Let $R$ be a commutative unital ring and $H$ is a Hopf algebra over $R$ which is a finitely generated projective $R$-module. Then,  $H^{\ast \ast} \cong H$. 
\end{rem}

By the above proposition, under some assumptions, the surjectivity of the trace map, i.e., $IS=S^\mathcal{A}$ is equivalent to the existence of the $\mathcal{A}$-module map $\mathcal{A}^\ast \to S$. 

Also, note that if $R$ is a principal ideal domain, then $J=\mathcal{A}^{\mathcal{A}}$ is free of rank one over $R$. Thus, we have the following. 

\begin{cor}
Let $R$ be a commutative local ring. Let $\mathcal{A}$ be a finite $R$-Hopf algebra and $S$ a faithful $\mathcal{A}$-module algebra which is a finite commutative $R$-algebra of $rank_R(S)=rank_R(\mathcal{A})$. Assume that $S^\mathcal{A}=R$. 

Then, if $R$ is a principal ideal domain, $S/R$ is a tame $\mathcal{A}$-extension if and only if the Hopfological homology of $S$ is equal to zero.   

If $R$ is a field $K$ and $S/K$ is a tame $\mathcal{A}$-extension,  the Hopfological homology of any relative $(S, \mathcal{A}^\ast)$-module is zero. 
\end{cor}

Similarly, the Hopfological homology for $\mathcal{A}^\ast$-comodules defined as in \cite{Fari} is zero in this case since it is zero for any injective comodules.

\begin{prop}
Let $R$ be a field and $S$ a finite faithful $H$-extension with a Hopf algebra $H$ over $R$. 
Suppose that $H$ is a local cocommutative Hopf algebra over $R$ and $rank_R(H)=rank_R(S)$. 
Then $S/R$ is a Hopf-Galois extension with $H$ if and only if the Hopfological homology of $S$ is equal to zero. 
\end{prop}
\begin{proof}
In this case $S/R$ is a Hopf-Galois extension with $H$ if and only if it is a tame $H$-extension by \cite[Theorem 14.7]{Childs1}. 






\end{proof}

The simplest example is to take $L/K$ to be Galois with Galois group $G$ and let $RG$ act on $S$. Then the condition, $IS=R$, is the same as the condition that the trace map $S \to R$ is surjective, which holds if and only if $L/K$ is tamely ramified field extension. 

\begin{rem}
Note that, for a discrete valuation ring $R$,  a finite Hopf algebra over $R$ is a Frobenius algebra. 
\end{rem}

\section{Hopfological homology and Hopf-cyclic homology}


We use the notation pseudo-para-cyclic object and cyclic object. These definitions are given in \cite{Ohara1} and \cite{Kaygun}. 
We proceed calculations in the setting of \cite{HKRS}.

\begin{rem}
If one would like to apply the approximation theorem as in \cite[Theorem 4.7, 4.8]{Kaygun},  one may assume that the antipode of $H$ is bijective and also assume that $S$ is commutative for satisfying the condition~\cite[p.347 (iii)]{Kaygun}. 

Since the comonad in \cite[Section 4]{Kaygun} consists of the scalar extension functor $H \otimes (-)$ and the forgetful functor, we have injective $H$-comodules after applying the approximation theorem, which is contructible in the derived category of $H$-comodules at least when $H$ is a coFrobenius Hopf algebra over a field. Therefore, we keep in mind the situation as follows; we let $\mathcal{H}$ be a Hopf algebra over a field $K$. Let $H < \mathcal{H}$ be a finite dimensional Hopf subalgebra. Let $L/K$ be an $H$-extension of fields, $R$ discrete valuation ring of $K$ and $S$ the integral closure of $R$ in $L$. We also feel like a subalgebra $\mathcal{A} \subset H$ to be the associated order of $S$. Note that $\mathcal{A}$ may not a Hopf algebra in general. 
\end{rem}

Now, we recall \cite[Definition 1.1]{HKRS} for the anti-Yetter-Drinfeld module over a Hopf algebra $H$, and use the (co)cyclic theory as in \cite[Section 2, Section 3]{HKRS}.  

We take a left-left stable anti-Yetter-Drinfeld module $M$ over $H$. Let $R$ be a commutative ring and $H$ a finite $R$-Hopf algebra. Let $S$ be an $R$-algebra which is a finitely generated faithful projective $R$-module and an $H$-module algebra. We use the tensor product over $R$. 

First, we define $T_n(S, M)$ by $S^{\otimes n} \otimes_R M$ and the face, degeneracy, and cyclic operators on it as in \cite{HKRS}; 
 \begin{align}\label{Hopf}
  d_{i}([a_{0}|a_{1}|\cdots|a_{n}]m) & =
  \begin{cases}
   [a_{0}|a_{1}| \cdots | a_{i}a_{i+1}| \cdots |  a_{n}]m , & i\neq n \\ 
     [a_n^{(0)}a_0 | a_{1} |\cdots |a_{n-1}]a_n^{(1)}m, & i=n 
  \end{cases} \\
  s_{i}([a_{0}|a_{1}|\cdots | a_{n}]m) & =
  [a_{0}|a_{1}| \cdots |a_{i}| 1 | a_{i+1} | \cdots | a_{n}]m \\
  t_{n}([a_{0}|a_{1}|\cdots | a_{n}]m ) & =
   [a_n^{(0)}|a_{0}|a_{1}| \cdots|a_{n-1}] a_{n}^{(1)}m. 
 \end{align}
These satisfy the condition $d_nt_n=t_{n-1}d_{n-1}$ but do not satisfy the invertibility and cyclicity of the operator $t_n$. 
Therefore, we need to take the cotensor product $S^{\otimes n} \square_H M$, i.e., we take $\Ker (\rho_{S^{\otimes n}} \otimes id_M - id_{S^{\otimes n}} \otimes \rho_M )$ where we denote by $\rho_{S^{\otimes n}}$ and $\rho_M$ the comodule structure maps, respectively, and we also need to take $M$ as left-left anti-Yetter-Drinfeld module over $H$ as in \cite{HKRS}. 
If we do so, $T_\bullet (S, M)$ inherits the cyclic structure. 

Now let us take $M'$ to be a relative $(S, H^\ast)$-module. Since the antipode of $H$ is bijective, we can associated the left $H^\ast$-comodule $M$ to the right $H$-comodule $M'$. Then, we can assume that $M$ is a stable left-left anti-Yetter-Drinfeld module over $H$. 

Before taking the cyclic module $S^{\otimes n} \square_H M$, we study the property of $T_\bullet (S, M)$.

Since the case of $H$-module algebras is written in \cite{HKRS}, one may see the following lemma and proposition. 

\begin{lemma}\label{213}
Let $R$ be a commutative ring and $H$ a finite $R$-Hopf algebra. Let $S$ be an $R$-algebra which is a finitely generated faithful projective $R$-module and an $H$-module algebra so that $j : S \# H \to \End_R(S)$ is an isomorphism. Then, for any left $S \# H$-module $M$, $M \cong S \otimes_R M^H$. 
\end{lemma}
\begin{proof}
Since $S$ is a faithful finitely generated projective $R$-module, there is a Morita equivalence from the category of left $R$-modules to the category of left $\End_R(S)$-modules given by $N \mapsto S \otimes_R N$ for an $R$-module $N$ and $M \mapsto \Hom_{\End_R(S)}(S, \End_R(S)) \otimes_{\End_R(S)} M$ for $M$ a left $\End_R(S)$-module. Thus, we have $M \cong S \otimes_R \Hom_{S \# H}(S, S \# H) \otimes_{S \# H} M$, and the right hand side is isomorphic to $\Hom_{S \# H}(S, M) \cong M^H$. 




\end{proof}

We see that if the following $\mathcal{A}$ become a Hopf algebra, then the degree shift of one-sided bar construction with coefficients in $M$ appears. 
We let $B_n(S, M)=S^{\otimes n}\otimes_R M$ and denote by $B_\bullet (S, M)$ the simplicial module obtained by one sided bar construction. 

\begin{prop}
Let $R$ be a commutative local ring. Let $\mathcal{A}$ be a finite $R$-Hopf algebra with bijective antipode and $S$ a faithful $\mathcal{A}$-module algebra which is a finite commutative $R$-algebra. 
Assume $S^\mathcal{A}=R$. 

Assume that $M$ is an $S \# \mathcal{A}$-module. Let us take one-sided bar construction of $S$ with coefficient in $M$ and $M^\mathcal{A}$ denoted by $B_\bullet(S, M)$ and $B_\bullet(S, M^\mathcal{A})$, respectively, where the tensor product is taken over $R$.  
If $j : S \# H \to E=End_R(S)$ is an isomorphism, then we have $B_\bullet(S, M) \cong B_\bullet(S, M^\mathcal{A})$, that is, the $B_\bullet(S, M)$ is isomorphic to the degree $+1$ shift of $B_\bullet(S, M^\mathcal{A})$. 
\end{prop}
\begin{proof}
In this setting, we have an isomorphism $S \# H \cong \End_R(S)$ if and only if we have $S \otimes_R S \cong S \otimes_R \mathcal{A}^\ast$, i.e., the condition of Hopf-Galois extension holds. 

By Lemma~\ref{213}, for any left $S \# \mathcal{A}$-module $M$, $M \cong S \otimes_R M^\mathcal{A}$. 
Therefore, we have $S^{\otimes n} \otimes_R M^{\mathcal{A}} \cong S^{\otimes (n-1)} \otimes_R M$ in each degree. 
\end{proof}

Together with \cite[Theorem 14.7]{Childs1}, which says that $S/R$ is Hopf-Galois extension with $\mathcal{A}$ if and only if it is a tame $\mathcal{A}$-extension under certain assumptions, we have the following. 
\begin{cor}
Let $R$ be a commutative local ring. Let $\mathcal{A}$ be a finite local cocommutative $R$-Hopf algebra with bijective antipode and $S$ a faithful $\mathcal{A}$-module algebra which is a finite commutative $R$-algebra of $rank_R(S)=rank_R(\mathcal{A})$. 
Assume $S^\mathcal{A}=R$. 

Assume that $M$ is an $S \# \mathcal{A}$-module. We denote by $B_\bullet(S, M)$ and $B_\bullet(S, M^\mathcal{A})$ the one-sided bar construction of $S$ with coefficient in $M$ and $M^\mathcal{A}$, respectively, where the tensor product is taken over $R$.  

Then, if the Hopfological homology of $S$ is equal to zero, then $B_\bullet(S, M) \cong B_\bullet (S, M^\mathcal{A})[1]$. 
\end{cor}
It seems that this corollary implies the negative degree of the bar construction $B_{-1}(S, M)$ seems like $B_0(S, M^\mathcal{A})$. 

Now, we consider the case $T_\bullet (S, M)$ with the relation (3.1), (3.2) and (3.3). 
\begin{prop}[case of comodule; due to the theorem of Doi and of Schneider]\label{shift}
Let $K$ be a field and $S$ an $\mathcal{H}$-module commutative algebra over $K$. For a finite Hopf subalgebra $H \subset \mathcal{H}$ with bijective antipode, assume that $S$ is a finite faithful $H$-extension of  $R=S^{coH^\ast}$ by restricting action of $\mathcal{H}$. 
Suppose that $S$ is a faithfully flat left $R$-module and $S/R$ is a Hopf-Galois extension with $H$. 
Then, 
 $T_\bullet(S, M) \cong T_\bullet(S, M^\mathcal{A})[1]$. 

\end{prop}
\begin{proof}
For a right $H^\ast$-comodule algebra $S$, a relative $(S, H^\ast)$-Hopf module $M \in \mathcal{M}_S^{H^\ast}$ is a left $S$-module in the monoidal category of $H^\ast$-comodules. 

For a finite dimensional Hopf algebra $H$ over a field $K$ and a right $H^\ast$-comodule algebra $S$, if there is a map $\phi : H^\ast \to S$ of right $H^\ast$-comodules algebras with $\phi (1)=1$, then the fundamental theorem for right relative $(S, H^\ast)$-Hopf modules is true~\cite[Chapter II, Theorem 3]{Doi2}.


More generally, let $H$ be a $R$-flat Hopf algebra, $S$ a right $H^\ast$-comodule algebra, and $R=S^{coH^\ast}$. Then the functor 
$\mathcal{M}_S^{H^\ast} \ni M \mapsto M^{coH^\ast} \in \mathcal{M}_R$ is right adjoint to $\mathcal{M}_R \ni N \mapsto N \otimes_R S \in \mathcal{M}_S^{H^\ast}$. 
Here, both the $S$-module and $H^\ast$-comodule structures of $N \otimes_R S$ are induced by those of $S$. 
A theorem of Schneider~\cite[Theorem 1]{Schneider} says that 
a faithfully flat left $R$-algebra $S$ is a Hopf-Galois extension of $R=S^{coH^\ast}$ with $H^\ast$ if and only if the adjunction is an equivalence. 

Therefore, we have $S \otimes_R M^{coH^\ast} \cong M$ by the fundamental theorem. 
Then, we have $S^{\otimes n} \otimes_R M^{coH^\ast} \cong S^{\otimes (n-1)} \otimes_R M$ in each degree, and the degree shift of $T_\bullet(S, M)$ appeared. 




\end{proof}

\bibliographystyle{amsplain} \ifx\undefined\bysame
\newcommand{\bysame}{\leavemode\hbox to3em{\hrulefill}\,} \fi

\end{document}